\documentclass[a4paper]{scrartcl}
\usepackage{amssymb, amsmath, amsthm, mathrsfs, pxfonts, braket}

\usepackage{lmodern}       

\usepackage{mathabx} 
\usepackage{graphicx}
\newcommand{\Conv}{\mathop{\scalebox{1.9}{\raisebox{-0.2ex}{$\Asterisk$}}}}%

\usepackage{hyperref}
\usepackage[hyperpageref]{backref} 
\usepackage{enumitem}

\usepackage[all,2cell]{xy}
\objectmargin+{0.5mm}
\labelmargin+{0.8mm}
\SelectTips{cm}{12}

\newtheorem{thm}{Theorem}[section]
\newtheorem{cor}[thm]{Corollary}
\newtheorem{lem}[thm]{Lemma}
\newtheorem{prop}[thm]{Proposition}

\theoremstyle{definition}
\newtheorem{dfn}[thm]{Definition}
\newtheorem{rem}[thm]{Remark}

\setlist[enumerate]{itemsep=0pt,label=$(\mathrm{\roman*})$, topsep=1pt}
\setlist[itemize]{itemsep=0pt, topsep=1pt, labelindent=\parindent,leftmargin=*}

\usepackage{etoolbox}
\usepackage{framed}
\newcommand\enclosebox[2]{%
  \BeforeBeginEnvironment{#1}{\begin{#2}}%
  \AfterEndEnvironment{#1}{\end{#2}}%
}

\enclosebox{thm}{leftbar}
\enclosebox{cor}{leftbar}
\enclosebox{lem}{leftbar}
\enclosebox{prop}{leftbar}
\enclosebox{dfn}{leftbar}
\enclosebox{rem}{leftbar}
\enclosebox{ex}{leftbar}

\newcommand{\Aut}{\operatorname{Aut}}
\newcommand{\ab}{\mathrm{ab}}
\newcommand{\abar}{\ol{a}}
\newcommand{\Ar}{\operatorname{Ar}}

\newcommand{\Cf}{\textit{cf.}\;}
\newcommand{\Cbar}{\overline{C}}
\newcommand{\Cpbar}{\overline{C'}}

\newcommand{\Cu}{\operatorname{Cu}}
\newcommand{\CuX}{\Cu(X)}
\newcommand{\CuXp}{\Cu(X')}
\newcommand{\Cov}{\operatorname{Cov}}

\newcommand{\Dif}{\mathfrak{D}}
\newcommand{\DifKp}{\Dif_{K'/K}}
\newcommand{\DifL}{\Dif_{L/K}}

\newcommand{\Div}{\operatorname{Div}^+}

\newcommand{\DivX}{\Div_{Z}(\Xbar)}
\newcommand{\DivXp}{\Div_{Z'}(\Xpbar)}

\newcommand{\fbar}{\ol{f}}
\newcommand{\fibar}{\ol{f_i\!}\,}
\newcommand{\fibarast}{\fibar^{\,\ast}}
\newcommand{\filD}{\operatorname{fil}_{D}}
\newcommand{\filDp}{\operatorname{fil}_{D'}}
\newcommand{\filDHX}{\filD\HX}
\newcommand{\filDpHX}{\filDp\HX}
\newcommand{\Fx}{\F_{\xbar}}

\newcommand{\fbarast}{\ol{f}^{\,\ast}}
\newcommand{\fbarastD}{\fbarast D}

\newcommand{\Fbar}{\overline{F}}

\newcommand{\F}{\mathscr{F}}
\newcommand{\gbarast}{\ol{g}^{\,\ast}}

\newcommand{\Gal}{\operatorname{Gal}}
\newcommand{\GL}{\operatorname{GL}}
\newcommand{\gbar}{\ol{g}}
\newcommand{\Hom}{\operatorname{Hom}}

\newcommand{\HX}{H^1(X,\QZ)}

\newcommand{\inj}{\hookrightarrow}
\newcommand{\isomto}{\stackrel{\simeq}{\to}}
\newcommand{\Ker}{\operatorname{Ker}}
\newcommand{\kU}{k(U)}

\newcommand{\kX}{k(X)}
\newcommand{\kXz}{\kX_z}
\newcommand{\kXbar}{\ol{\kX}}
\newcommand{\kXzbar}{\ol{\kXz}}
\newcommand{\Kbar}{\overline{K}}
\newcommand{\kbar}{\overline{k}}
\newcommand{\op}{\mathrm{op}}

\newcommand{\onto}[1]{\stackrel{#1}{\to}}
\newcommand{\ol}[1]{\overline{#1}}
\newcommand{\pbar}{\ol{p}}
\newcommand{\phiast}{\phi^{\ast}}
\newcommand{\phibar}{\overline{\phi}}
\newcommand{\phipbar}{\ol{\phi'}}

\newcommand{\phiastF}{\phiast\F}
\newcommand{\phibarast}{{\overline{\phi}}^{\,\ast}}
\newcommand{\phipbarast}{{\overline{\phi'}}^{\,\ast}}
\newcommand{\phibarastD}{\phibarast\!D}

\newcommand{\piCa}{\pi_1(C,\ol{a})}
\newcommand{\piCDa}{\pi_1(C,\phibarastD, \ol{a})}

\newcommand{\piU}{\pi_1(U)}
\newcommand{\piX}{\pi_1(X)}
\newcommand{\piXx}{\pi_1(X,\xbar)}
\newcommand{\piXxi}{\pi_1(X,\xibar)}
\newcommand{\piXD}{\pi_1(X,D)}
\newcommand{\piXDx}{\pi_1(X,D,\xbar)}
\newcommand{\piXDxi}{\pi_1(X,D,\xibar)}

\newcommand{\piXxSets}{(\piXx\mbox{-sets})} 
\newcommand{\piXpxSets}{(\pi_1(X',\xpbar)\mbox{-{sets}})}

\newcommand{\piXbarx}{\pi_1(\Xbar,\xbar)}

\newcommand{\piXDabg}{\pi_1(X,D)^{\ab,0}}
\newcommand{\piXab}{\pi_1(X)^{\ab}}
\newcommand{\piXDab}{\pi_1(X,D)^{\ab}}

\newcommand{\Ql}{\Q_l}
\newcommand{\QZ}{\Q/\Z}
\newcommand{\Qlbar}{\overline{\Ql\!}\,}
\newcommand{\red}{\mathrm{red}}
\newcommand{\Rr}{\mathcal{R}_r}
\newcommand{\RrX}{\Rr(X)}
\newcommand{\rhobar}{\overline{\rho}}
\newcommand{\Sr}{\mathcal{S}_r}
\newcommand{\Sw}{\operatorname{Sw}}
\newcommand{\SwV}{\Sw(V)}

\newcommand{\SwF}{\Sw(\F)}
\newcommand{\SwzF}{\Sw_z(\F)}
\newcommand{\ssm}{\smallsetminus}
\newcommand{\Spec}{\operatorname{Spec}}
\newcommand{\surj}{\twoheadrightarrow}

\newcommand{\ul}[1]{\underline{#1}}

\newcommand{\Xred}{X_{\red}}
\newcommand{\Xbar}{\overline{X}}
\newcommand{\Xbarred}{\Xbar_{\red}}
\newcommand{\Xpbar}{\overline{X'}}
\newcommand{\xbar}{\overline{x}}
\newcommand{\xpbar}{\overline{x'}}
\newcommand{\xibar}{\overline{\xi}}

\newcommand{\Zhat}{\widehat{\Z}}

\newcommand{\WXx}{W(X,\xbar)}

\newcommand{\wt}[1]{\widetilde{#1}}

\newcommand{\R}{\mathbb{R}}
\newcommand{\Z}{\mathbb{Z}}
\newcommand{\Q}{\mathbb{Q}}

\newcommand\sn{\smallskip\noindent}

\title{A Hermite-Minkowski type theorem of varieties over finite fields}
\author{Toshiro Hiranouchi}

\begin{document}
\pagenumbering{arabic}
\maketitle

\begin{abstract}
We show  
the finiteness of \'etale coverings of a variety over a finite field  
with given degree 
whose ramification bounded along an effective Cartier divisor. 
The proof is an application of P.~Delgine's theorem \cite{EK} 
(H.~Esnault and M.~Kerz in Acta Math.\ Vietnam.\ 37:531--562, 2012) 
on a finiteness of $l$-adic sheaves 
with restricted ramification.  
By applying our result to a smooth curve over a finite field, 
we obtain a function field analogue of the classical Hermite-Minkowski theorem.
\end{abstract}

\section{Introduction}
For a number field $F$, that is, a finite extension of $\Q$, 
the Hermite-Minkowski theorem 
asserts that 
there exist only finitely many extensions of the number field $F$ 
with given degree unramified outside a finite set of primes of $F$ 
(e.g., \cite{N}, Chap.~III, Thm.~2.13; \cite{F}, Chap.~V, Thm.~2.6). 
In \cite{F}, 
G.~Faltings gave  
a higher dimensional generalization of this theorem stated as follows:

\begin{thm}[\cite{F}, Chap.~VI, Sect.~2.4; \cite{HH}, Thm.~2.9]
\label{thm:Faltings}
Let 
$X$ be a connected scheme of finite type and dominant (e.g., flat) over\/ 
$\Spec(\Z)$. 
Then there exist only finitely many \'etale coverings of $X$ 
with given degree. 
\end{thm}

\noindent
Here, an \textbf{\'etale covering} of $X$ means a finite \'etale morphism $X'\to X$. 
The aim of this note is to give a ``function field'' analogue of this theorem. 
We begin simple observations:

\begin{itemize}
\renewcommand{\labelitemi}{$\circ$}
  \item 
  For a function field  $F$  
  of one variable over a finite field with characteristic $p$, 
  the Artin-Schreier equations produce infinitely many extensions
  of $F$ of degree $p$ which ramify only in a finite set of places (see e.g., \cite{G}, Sect.~8.23). 
  \item For a number field $F$,  
  (the exponents of) the discriminant of an extension of $F$ has an upper bound 
  depending on the extension degree and the primes at which it ramifies 
(\cite{SerCL}, Chap.~III, Sect.~6, Prop.~13, see also remarks after the proposition). 
Under the conditions in the Hermite-Minkowski theorem, namely, the extension degree and 
a finite set of primes are given,   
the discriminants of extensions of $F$ are automatically bounded. 
\end{itemize}
Considering these facts together, 
to obtain a finiteness as above in the case of function fields 
we have to restrict ramification.

Now, we present the results in this note more precisely.  
Let 
$X$ be a connected and separated scheme of finite type over a finite field 
(we call such schemes just \textbf{varieties} in the following 
\Cf Notation), and 
$\Xbar$ a compactification of $X$ (\Cf Sect.~\ref{sec:ram}). 
For an effective Cartier divisor $D$ with support 
$|D| \subset Z = \Xbar \ssm X$, 
we will introduce the notion of 
\emph{bounded ramification along $D$} for \'etale coverings of $X$ 
(whose ramification locus is in the boundary $Z$) 
in the next section (Def.~\ref{def:rambdd}). 
Adopting this notion, we show the following theorem.

\begin{thm}[Thm.~\ref{thm:main}]
\label{thm:intro}
Let $X\subset \Xbar$ be as above. 
There exist only finitely many 
\'etale coverings of\/ $X$ with bounded degree and ramification bounded by 
a given effective Cartier divisor $D$ with support in $Z=\Xbar\ssm X$.  
\end{thm}

\noindent
A key ingredient for the proof is (a weak form of) 
Deligne's finiteness theorem 
on smooth Weil sheaves with bounded ramification  
\cite{EK} (Thm.~\ref{thm:Deligne}).  

For an \'etale covering $X' \to X$ 
of smooth \emph{curves} over a finite field, 
if its  degree and the discriminant  
are bounded, then 
the ramification of the covering $X' \to X$ in our sense 
is also bounded (Prop.~\ref{prop:disc}): 
\begin{center}
	 bounded degree \& discriminant $\Rightarrow$ bounded ramification. 
\end{center}
From this, 
we obtain an alternative proof of the following well-known theorem: 

\begin{cor}[\cite{G}, Thm.~8.23.5]
\label{thm:Taguchi}
Let 
$F$ be a function field of one variable over a finite field. 
Then 
there exist only finitely many separable extensions of $F$ 
with bounded degree and discriminant. 
\end{cor}

\noindent
On the other hand, we have another implication 
\begin{center}
	 bounded degree \& ramification $\Rightarrow$ bounded discriminant
\end{center}
(see Rem.~\ref{rem:disc}). 
As a result, Thm.~\ref{thm:Taguchi} is equivalent to 
the main theorem (Thm.~\ref{thm:intro}) 
for the case where $X$ is a smooth curve over a finite field.

\subsection*{Contents} 
The contents of this note is the following: 

\begin{itemize}
\item Sect.~\ref{sec:ram}\,:  

\begin{itemize}
\item   
We define the notion of 
bounded ramification along an effective Cartier divisor $D$ 
for \'etale coverings of a variety $X$ over a finite field (Def.~2.1). 
We also introduce the fundamental group $\piXD$ 
which classifies such \'etale coverings of $X$ 
with bounded ramification along $D$ 
(Def.~2.2).

\item 
We define the Swan conductor of smooth $\Qlbar$-sheaves 
as in \cite{EK} (Def.~\ref{def:Swan}). 
In Lem.~\ref{lem:bdd}, we give a relation between 
the Swan conductor of smooth $\Qlbar$-sheaf and 
our notion of the bounded ramification.  
\end{itemize}

\item 
Sect.~\ref{sec:small}\,:     

\begin{itemize}
\item We interpret the Hermite-Minkowski type finiteness as 
above into a property of profinite groups called \emph{smallness} 
which is studied in \cite{HH} (Def.~3.1). 

\item The proof of the main theorem (Thm.~\ref{thm:main}) is given by 
showing the smallness of the fundamental group $\piXD$. 

\item We also provide some applications of our main theorem to 
a finiteness of representations (with finite images) 
of the fundamental group $\piXD$  (Cor.~3.6). 
\end{itemize}
\end{itemize}

\subsection*{Notation}
In this note, a \textbf{local field} is a complete discrete valuation field 
with perfect residue field. 
For such a local field $K$, we denote by 

\begin{itemize}
	\item $O_K$\,: the valuation ring of $K$, and 
	\item $v_K$\,: the valuation of $K$ normalized as $v_K(K^{\times}) =\Z$. 
\end{itemize}
For a field $F$, we denote by 

\begin{itemize}
\item $\Fbar$\,: a separable closure of $F$, and 
\item $G_F:= \Gal(\Fbar/F)$\,: the Galois group of $\Fbar$ over $F$.
\end{itemize}
We also use the following notation:

\begin{itemize} 
\item $p$\,: a fixed prime number, 
\item $l$\,: a prime number $\neq p$, and 
\item $k$\,: a finite field of characteristic $p$.
\end{itemize}
Throughout this note, we assume that $l$ is invertible in all schemes we consider. 
For a scheme $X$, 
an \textbf{\'etale covering} of $X$ we mean a finite \'etale morphism $X'\to X$. 
A \textbf{variety} over $k$ 
means a separated and connected scheme of finite type over $\Spec(k)$. 
A \textbf{curve} over $k$ is a variety over $k$ with dimension $1$. 
For an integral variety $X$ over $k$, we denote by 

\begin{itemize}
\item $k(X)$\,: the function field of $X$. 
\end{itemize}

\noindent 
Following \cite{Fu}, Sect.~3.2, 
we call a pair $(X,\xbar)$ of  a scheme and a geometric point $\xbar$ of $X$ 
a \textbf{pointed scheme}. 
A \textbf{morphism} $f:(X',\xpbar) \to (X,\xbar)$ of pointed schemes 
means a morphism $f:X'\to X$ of schemes with $f(\xpbar) = \xbar$.

\subsection*{Acknowledgments}

The author would like to thank Professor Yuichiro Taguchi 
for helpful advice and comments on Lem.~\ref{lem:bdd},   
and the proof of Thm.~\ref{thm:main}.  
Thanks also are due to Shinya Harada, and Takahiro Tsushima 
for fruitful discussions and valuable suggestions 
on this note. 
This work was supported by JSPS KAKENHI Grant Number 25800019.

\section{Ramification}
\label{sec:ram}
Adding to Notation, 
throughout this note, 
we use the following notation: 
For a variety $X$ over a finite field $k$ (\Cf Notation), 

\begin{itemize} 
\item $\Xbar$\,: a compactification of $X$ (over $\Spec(k)$) 
by Nagata's theorem (\Cf \cite{Lue}), 
that is, a proper scheme over $\Spec(k)$ which contains $X$ as a dense open subscheme,   

\item $Z:=\Xbar \ssm X$\,: the boundary of $X$,  

\item $\CuX$\,: the set of the normalizations of closed integral subschemes 
of $X$ of dimension $1$, and

\item $\DivX$\,: the monoid of 
effective Cartier divisors $D$ on $\Xbar$ whose support $|D|$ is in $Z$.    
\end{itemize} 

\noindent
For each $\phi:C\to X \in \CuX$, 
we denote by 

\begin{itemize} 
\item $\Cbar$\,: the smooth compactification of $C$ (which exists uniquely 
by a resolution of singularities),

\item $\phibar:\Cbar \to \Xbar$\,: the canonical extension of $\phi$ (by 
the valuative criterion of properness), and

\item $k(C)_x$\,: the completion of the function field $k(C)$ at $x \in \Cbar$  
which is a local field in our sense. 
\end{itemize}

\noindent
For a finite Galois extension $L/K$ of local fields, 
we use the following ramification filtrations (\Cf \cite{SerCL}, Chap.~IV, Sect.\ 3): 

\begin{itemize}
\item $(\Gal(L/K)_{\mu})_{\mu\in \R_{\ge -1}}$\,: the ramification filtration of $\Gal(L/K)$ 
  in the lower numbering which is given by 
\begin{equation}
	\label{eq:lower}
  \Gal(L/K)_{\mu} = \set{\sigma \in \Gal(L/K) | v_L(\sigma(\theta) -\theta) \ge \mu + 1},
\end{equation}
where $\theta \in L$ is a generator of the valuation ring $O_L$ as 
an $O_K$-algebra: $O_L = O_K[\theta]$ (\cite{SerCL}, Chap.~III, Sect.~6, Prop.~12). 
For each $\mu \in \R_{\ge -1}$, $\Gal(L/K)_{\mu}$ is a normal subgroup of $\Gal(L/K)$. 
In particular, it is known that  
$\Gal(L/K)_1$ is the $p$-Sylow subgroup of the inertia subgroup $\Gal(L/K)_0$, 
where $p$ is the characteristic of the residue field of $K$.
 
\item $(\Gal(L/K)^{\lambda})_{\lambda\in \R_{\ge -1}}$\,: the ramification filtration of $\Gal(L/K)$ in the upper numbering which is defined by the relation 
\begin{equation}
	\label{eq:upper}
   \Gal(L/K)^{\varphi_{L/K}(\mu)} = \Gal(L/K)_{\mu}, 
\end{equation}
where $\varphi_{L/K}:\R_{\ge -1} \to \R_{\ge -1}$ is the {Herbrand function} 
\begin{equation}
  \label{eq:Herbrand}	
 \varphi_{L/K}(\mu ) = \int^{\mu}_0 \frac{d x}{(\Gal(L/K)_0:\Gal(L/K)_x)}
\end{equation}
and its inverse function is denoted by $\psi_{L/K}$. 
\end{itemize}

\noindent
We note here, these ramification filtrations satisfy 
the following compatibility properties: 

\begin{lem}[\cite{SerCL}, Chap.~IV, Sect.~1, Prop.~2, and Sect.~3, Prop.~14]
\label{lem:fil}
Let $L/K$ be a finite Galois extension of local fields. 

\begin{enumerate}
	\item  For a sub extension $K'/K$ of $L$, we have 
	\[
		\Gal(L/K')_{\mu} = \Gal(L/K)_{\mu} \cap \Gal(L/K'), 
	\] 
	for any $\mu\in \R_{\ge -1}$.
	\item For a sub Galois extension $K'/K$ of $L$, we have  
	\[
		\Gal(K'/K)^{\lambda} = \Gal(L/K)^{\lambda}\Gal(L/K')/\Gal(L/K'),
	\]
	for any $\lambda \in \R_{\ge -1}$.
\end{enumerate} 
\end{lem}
\noindent
For a local field $K$, 
from Lem.\ \ref{lem:fil} (ii) one can introduce  

\begin{itemize} 
\item $(G_K^{\lambda})_{\lambda \in \R_{\ge -1}}$\,: the ramification filtration 
of $G_K = \Gal(\Kbar/K)$ (in the upper numbering) 
which is given by 
\begin{equation}
\label{eq:G_K^l}
G_K^{\lambda} = \varprojlim_{\mbox{\tiny{$L/K$:\,finite Galois $ \subset \Kbar$}}} \Gal(L/K)^{\lambda}, 
\end{equation}
and  

\item $G_K^{\lambda+}:=$ the topological closure of 
  $\bigcup_{\lambda'>\lambda}G_K^{\lambda'}$ in $G_K$ for $\lambda \in \R_{\ge 
  0}$. 
\end{itemize}
\noindent

\subsection*{\'Etale coverings} 

\begin{dfn}[\Cf \cite{HH}, Def.~3.2]
\label{def:rambdd}
\begin{enumerate}
\item
Let 
$K$ be a local field.   
For a separable field extension $L/K$ (contained in $\Kbar$) and 
$\lambda \in \R_{\ge 0}$, 
we say that the \textbf{ramification of $L/K$ is bounded by $\lambda$} 
if we have 
$G_K^{\lambda+} \subset G_L = \Gal(\Kbar/L)$. 

\item
Let $X$ be a smooth \emph{curve} over $k$, 
and $\Xbar$ the smooth compactification of $X$. 
Let $D = \sum_{z\in Z} m_z [z] \in \DivX$ ($m_z \in \Z_{\ge 0}$).    
An \'etale covering $X'\to X$ is said to be 
of \textbf{ramification bounded by $D$} 
if the extension
$k(X')_{z'}/ k(X)_z$  of local fields 
is of ramification bounded by $m_z$ 
for all $z \in Z$ (putting $m_z =0$ if $z \not \in |D|$) and for all 
$z'\in \overline{X'}$ above $z$.

\item
Let $X$ be a \emph{variety} over $k$. 
For $D \in \DivX$,  
an \'etale covering $X' \to X$ is said to be 
of \textbf{ramification bounded by $D$} 
if for every $\phi:C\to X \in \CuX$ and for each irreducible component $C'$ 
of $C\times_XX'$,  
the ramification of the induced morphism $C' \to C$ is bounded by 
$\phibarastD$, 
where $\phibarastD$ is the inverse image of $D$ by $\phibar:\Cbar \to \Xbar$ 
(for the existence of $\phibarastD$, see Lem.~\ref{lem:divisor} below). 
\end{enumerate}
\end{dfn}

\begin{lem}
\label{lem:fund}
Let $L/K$ be a finite separable extension of local fields. 
\begin{enumerate}
	\item 
The extension  $L/K$ is 
tamely ramified if and only if it is of ramification bounded by $0$.

\item 
    We denote by $\wt{L}$ the Galois closure of $L/K$. 
For any $\lambda \in \R_{\ge 0}$, 
the extension $L/K$ is of ramification bounded by $\lambda$ if and only if 
$\wt{L}/K$ is of ramification bounded by $\lambda$. 

\item 
Assume that the extension $L/K$ is Galois. 
For any $\lambda\in \R_{\ge 0}$, 
the ramification of $L/K$ is bounded by $\lambda$ 
if and only if $\Gal(L/K)^{\lambda'} = \set{1}$ for any $\lambda'>\lambda$.  
\end{enumerate}
\end{lem}
\begin{proof}
(ii) 
The ``if'' part follows immediately from the definition. 
We show the ``only if'' part.
Assume that $L/K$ is of ramification bounded by $\lambda$ for some 
$\lambda \in \R_{\ge 0}$. 
Namely, we have $G_K^{\lambda+} \subset G_L$.  
Since $G_{\wt{L}} = \Gal(\Kbar/\wt{L})$ is the maximal normal subgroup of $G_K$ which is contained in $G_{L}$, we have $G_{K}^{\lambda+} \subset G_{\wt{L}}$.

\sn 
(iii) 
The restriction $G_K \surj \Gal(L/K); \sigma \mapsto \sigma|_L$  
induces 
$G_K^{\lambda} \surj \Gal(L/K)^{\lambda}$ from the very definition of $G_K^{\lambda}$ in \eqref{eq:G_K^l} for any $\lambda$. 
For any $\lambda' > \lambda$, 
$G_K^{\lambda'} \subset G_L$ if and only if $\Gal(L/K)^{\lambda'} = \set{1}$, 
and the assertion follows from it. 

\sn 
(i) 
By taking the Galois closure of the extension $L/K$ and using (ii), 
we may assume that $L/K$ is a Galois extension. 
Let $p$ be the characteristic of 
the residue field of $K$. 
As $\bigcap_{\lambda>0}\Gal(L/K)^{\lambda} = \Gal(L/K)_1$ 
is the $p$-Sylow subgroup of the inertia subgroup $\Gal(L/K)^0$, 
the extension 
$L/K$ is tamely ramified if and only if $\Gal(L/K)^{\lambda} = \set{1}$ for any $\lambda>0$. 
The assertion (i) follows from (iii).
\end{proof}

For a pointed connected Noetherian scheme $(X,\xbar)$ (\Cf Notation), 
we define

\begin{itemize}
\item $\Cov(X)$\,: the Galois category of \'etale coverings of $X$, 

\item $\piXx$\,: the fundamental group of $(X,\xbar)$ associated to $\Cov(X)$, 
which is defined by 
\[
   \piXx = \varprojlim_{(X', \ol{x'}) \in \Cov(X)} \Aut_X(X')^{\op},
\]
where the projective limit is taken over a projective system of pointed Galois 
coverings  $(X', \xpbar) \to (X,\xbar)$ in $\Cov(X)$
(\cite{SGA1}, Exp.~V, Sect.~7), and  

\item 
$\piXxSets$\,: the category of 
finite sets on which $\piXx$ acts continuously on the left.   
\end{itemize}

\noindent 
The fiber functor $Y \mapsto \Hom_X(\xbar, Y)$ gives an 
equivalence of categories 
\begin{equation}
\label{eq:Gal}
\xymatrix@C=5mm{
  \Cov(X) \ar[r]^-{\simeq} & \piXxSets
  }
\end{equation}
(\cite{SGA1}, Exp.~V, Prop.~5.8; \cite{Fu}, Thm.~3.2.12). 

Now, we assume that $X$ is a variety over $k$ (\Cf Notation). 
For $D\in \DivX$, 
the full subcategory $\Cov(X\subset \Xbar,D) \subset \Cov(X)$ 
of \'etale coverings of $X$ with 
ramification bounded by $D$ 
also forms a Galois category  (see \cite{HH}, Lem.~3.3). 

\begin{dfn}
\label{def:piXD}
Associated to the Galois category $\Cov(X\subset \Xbar,D)$, 
we define the fundamental group by 
\[
\pi_1(X\subset \Xbar, D, \xbar) = \varprojlim_{(X', \xpbar) \in \Cov(X\subset \Xbar, D)} \Aut_X(X')^{\op},
\] 
where the projective limit is taken over a projective system of pointed Galois 
coverings $(X', \xpbar) \to (X,\xbar)$ in $\Cov(X\subset \Xbar, D)$. 
\end{dfn}

In the following, 
we write $\piXDx$ and $\Cov(X,D)$ 
when we need not specify the fixed compactification $\Xbar$ of $X$. 
For the geometric point $\xbar:\Spec(\Omega) \to X$ with 
a separably closed field $\Omega$, 
we denote also by $\xbar$ 
the geometric point $\Spec(\Omega) \onto{\xbar} X \inj \Xbar$ of $\Xbar$. 
The functors 
$\Cov(\Xbar) \to \Cov(X,D);\ \overline{Y} \mapsto \overline{Y} \times_{\Xbar} X$, and 
$\Cov(X,D) \to \Cov(X);\ Y \mapsto Y$ 
induce homomorphisms 
\begin{equation}
\label{eq:pi}
  \xymatrix@C=5mm{
  \piXx \ar@{->>}[r] & \piXDx \ar@{->>}[r] & \piXbarx
  }
\end{equation}
which are surjective 
(\cite{SGA1}, Exp.~V, Prop.~6.9).

\begin{lem}
\label{lem:pull-back}
Let $f:(X',\xpbar) \to (X,\xbar)$ be a morphism of pointed varieties over $k$,  
and $D\in \DivX$. 
Assume the conditions $\mathrm{(a)}$ and $\mathrm{(b)}$ below. 

\begin{enumerate}[label=$\mathrm{(\alph*)}$]
\item 
There exists a commutative diagram:  
\[
	\xymatrix@C=15mm{
	\Xpbar\ar[r]^{\fbar} & \Xbar \\
	X' \ar@{^{(}->}[u]\ar[r]^f & X\ar@{^{(}->}[u]
	}
\]
with a morphism $\fbar:\Xpbar \to \Xbar$ 
from a compactification $\Xpbar$ of $X'$ to $\Xbar$, 
where the vertical morphisms are the inclusions.

\item 
The inverse image 
$\fbarastD \in \DivXp$ 
of $D$ exists, where $Z' = \Xpbar \ssm X'$.
\end{enumerate}

\noindent
Then 
we have a canonical homomorphism 
\[
  \xymatrix@C=5mm{
	\pi_1(X'\subset \Xpbar, \fbarastD, \xpbar) \ar[r]& \pi_1(X\subset \Xbar, D,\xbar).
	}
\]
\end{lem}
\begin{proof}
We prove that the functor 
$\Cov(X) \to \Cov(X');\ Y \mapsto Y\times_X X'$ 
induces a functor $\Cov(X\subset \Xbar ,D) \to \Cov(X'\subset \Xpbar, \fbarastD)$.
The latter functor gives the required homomorphism 
on the fundamental groups (\cite{SGA1}, Exp.~V, Sect.~6). 
For $h:Y\to X \in \Cov(X\subset \Xbar ,D)$,  
it is enough to show that 
$Y' := Y\times_X X' \to X'$ 
is of ramification bounded by $\fbarastD$. 
Take any $\phi':C' \to X' \in \CuXp$, and we have to show 
that $h':Y'\times_{X'} C' \to C'$ 
gives \'etale coverings of ramification bounded by $\phipbarast(\fbarastD)$. 
Here, we divide the proof into the two cases according to the image 
$f \circ \phi'(C')$ in $X$. 

\sn
(\textbf{The case where $f\circ \phi'(C')$ is a point})
If the image of 
the composite $f \circ \phi'$ is a closed point of $X$, 
then $h':Y'\times_{X'} C' \to C'$ induces 
separable constant field extensions of $k(C')$, 
and thus unramified on the boundary of $C'$. 
In particular, $h'$ is of ramification bounded by $\phipbarast(\fbarastD)$. 

\sn
(\textbf{The case where $f\circ \phi'(C')$ is a curve})
If $f\circ \phi'$ factors through 
$\phi:C\to X\in \CuX$ as in 
\[
	\xymatrix@C=15mm{
	C'\ar[r]^{\phi'}\ar[d]_{g} & X'\ar[d]^f \\
	C \ar[r]^{\phi} & X,
	}
\]
then 
the extension $\gbar:\Cpbar \to \Cbar$ of $g$ fits into 
the following commutative diagram: 
\[
	\xymatrix@C=15mm{
	\Cpbar \ar[d]_{\gbar} \ar[r]^{\phipbar} & \Xpbar \ar[d]^{\fbar}\\
	\Cbar \ar[r]^{\phibar} & \Xbar .
}
\]   
Since $h:Y\to X$ is of ramification bounded by $D$, 
the base change $Y \times_X C \to C$ produces 
\'etale coverings of $C$ which are of ramification bounded by $\phibarastD$. 
It is left to show that 
$h':Y'\times_{X'}C' = Y\times_X C' \to C'$ 
is of ramification bounded by $\phipbarast(\fbarastD) =\gbarast(\phibarastD)$.
As a result, we may assume that  
$f:X'\to X$ is a 
dominant morphism of smooth curves over $k$. 

Let $D = \sum_z m_z [z] \in \DivX$, and    
$Y \in \Cov(X\subset \Xbar ,D)$. 
Take an irreducible component 
$Y'$ of $Y\times_X X'$.  
We prove that the induced covering 
$Y'\to X'$ is of ramification bounded by $\fbarastD$. 
At $z' \in Z' = \Xpbar \ssm X'$ with $\ol{f}(z') = z$, 
the multiplicity of $\fbarastD$ is 
$e_z m_z$, 
where $e_z$ is the ramification index of the extension $k(X')_{z'}/k(X)_z$. 
The assertion carries over to the local situation as in 
the following lemma.  
\end{proof}

\begin{lem} 
\label{lem:em}
Let $K$ be a local field, and  
$K'/K$ a finite extension 
with ramification index $e$.   
If a finite separable extension $L/K$ is of ramification bounded by $\lambda$ for some $\lambda\in \R_{\ge 0}$, 
then $LK'/K'$ is of ramification bounded by $e\lambda$. 
\end{lem} 
\begin{proof}
\textbf{(Reduction to $L/K$ is Galois)}
We denote by $\wt{L}$ the Galois closure of $L/K$. 
From Lem.\ \ref{lem:fund} (ii), 
$\wt{L}/K$ is of ramification bounded by $\lambda$. 
On the other hand, if $\wt{L}K'/K'$ is of ramification bounded by $e\lambda$, so is $LK'/K'$ 
by Lem.\ \ref{lem:fund} (ii) again. 
We may assume that $L/K$ is a Galois extension. 

%

\sn
\textbf{(Proof of the lemma)}
We put $L' = LK'$ 
and consider the two cases below: 

\sn
\textbf{Case 1 ($K'/K$ is purely inseparable):} 
Assume that $K'/K$ is purely inseparable. 
In this case, we show that $L'/K'$ is of ramification bounded by $\lambda$. 
Since we have $K'\cap L = K$, 
the restriction gives an isomorphism 
$\Gal(L'/K') \isomto \Gal(L/K)$ and hence 
$[K':K] = [L':L] = e$. 
As the local field $K$ is of characteristic $p>0$ 
with perfect residue field, 
we have $(K')^e = K$ and $(L')^e = L$.
Take $\theta \in O_{L'}$ such that $O_{L'} = O_{K'}[\theta]$ 
(\cite{SerCL}, Chap.~III, Sect.~6, Prop.~12) and this gives 
$O_{L} = O_K[\theta^e]$. 
For any $\sigma \in \Gal(L'/K')$, we have 
\begin{align*}
	e\cdot  v_{L'}(\sigma(\theta) - \theta) 
	&= v_{L'}(\sigma(\theta^e) - \theta^e) \quad  (\mbox{since $e$ is a power of $p$})\\
	&= e\cdot v_{L}(\sigma(\theta^e) - \theta^e). 
\end{align*}
Thus, the restriction $\Gal(L'/K) \isomto \Gal(L/K)$ induces an isomorphism 
$\Gal(L'/K')_{\mu} \simeq \Gal(L/K)_{\mu}$ for any $\mu\in \R_{\ge -1}$ 
and thus $\psi_{L'/K'} = \psi_{L/K}$ by \eqref{eq:Herbrand}.  
From the assumption, for any $\lambda'>\lambda$, 
\[
\Gal(L'/K')^{\lambda'} 
= \Gal(L'/K')_{\psi_{L'/K'}(\lambda')} 
\simeq \Gal(L/K)_{\psi_{L/K}(\lambda')} = \Gal(L/K)^{\lambda'} = \set{1}. 
\]
This implies that the extension $L'/K'$ is of ramification bounded by $\lambda$ (Lem.\ \ref{lem:fund} (iii)).

\sn 
\textbf{Case 2 ($K'/K$ is separable):} 
  Assume that $K'/K$ is a separable extension. 
  We show that $L'/K'$ is of ramification bounded by $e\lambda$. 
  Take a finite Galois extension $M/K$ with $L' \subset M$. 
  The Herbrand function $\psi_{K'/K}$ of the separable extension $K'/K$ is 
  defined to be 
  $\psi_{K'/K}= \varphi_{M/K'} \circ \psi_{M/K}$ 
  (\cite{SerCL}, Chap.~IV, Sect.~3, Rem.~2) and it satisfies 
  $\psi_{K'/K}(\mu/e) \le \mu$ for any $\mu\in \R_{\ge -1}$. 
 In fact,  
\[
  \varphi_{K'/K}(\mu) = \varphi_{M/K} \circ \psi_{M/K'}(\mu)
  \stackrel{(\dagger)}{\ge} \frac{1}{e}\varphi_{M/K'}\circ \psi_{M/K'}(\mu)
     = \frac{\mu}{e},  
\]
  where the inequality  ($\dagger$)
  follows from $\#\Gal(M/K)_0 = e \#\Gal(M/K')_0$ and Lem.\ \ref{lem:fil} (i). 
  For any $\lambda' >e\lambda$, we have 
  \begin{align*}
  	\Gal(M/K')^{\lambda'} 
  	&\subset \Gal(M/K')^{\psi_{K'/K}(\lambda'/e)}\quad (\mbox{by $\psi_{K'/K}(\lambda'/e) \le \lambda'$})\\
  	&= \Gal(M/K')_{\psi_{M/K}(\lambda'/e)} \quad (\mbox{by Lem.\ \ref{lem:fil} (ii)})\\
  	&= \Gal(M/K)_{\psi_{M/K}(\lambda'/e)}\cap \Gal(M/K') \quad (\mbox{by Lem.\ \ref{lem:fil} (i)}) \\
  	& \subset \Gal(M/K)^{\lambda'/e}\quad (\mbox{by Lem.\ \ref{lem:fil} (ii)})\\
  	& \subset \Gal(M/L)\quad (\mbox{by $\lambda'/e >\lambda$ and $L/K$ is of ramification bounded by $\lambda$}).
  \end{align*}
  As a result, we have 
  $\Gal(M/K')^{\lambda'} \subset \Gal(M/L)\cap \Gal(M/K') = \Gal(M/L')$ 
  and hence 
  \[
  \Gal(L'/K')^{\lambda'} 
    = \Gal(M/K')^{\lambda'}\Gal(M/L')/\Gal(M/L') = \set{1}
   \] 
   by  Lem.\ \ref{lem:fil} (ii).
   This implies the assertion by Lem.\ \ref{lem:fund} (iii).

\sn
\textbf{(Proof of the lemma -- continued)} 
We take the separable closure $K^s$ of $K$ within $K'$. 
From Case 2 above,  
the extension 
$LK^s/K^s$ is of ramification bounded by $e\lambda$. 
Applying Case 1 to the purely inseparable extension $K'/K^s$, 
the extension $L' = LK'/K'$ is of ramification bounded by $e\lambda$ as required.  
\end{proof}

Concerning the condition (b) in Lem.~\ref{lem:pull-back}, 
the lemma below assures  
the existence of the inverse image of an effective Cartier divisor 
in some situations.

\begin{lem}[\Cf \cite{EGA}, Chap.~IV, Sect.~21.4]
\label{lem:divisor}
Let $f:X'\to X$ be a morphism of varieties over $k$, 
and $D$ an effective Cartier divisor on $X$. 
Then the inverse image $f^{\ast}D$ is defined in each of the 
following conditions:

\begin{enumerate}[label=$\mathrm{(\alph*)}$]
\item
$f$ is flat, 
\item
$X',X$ are integral, and  
$f$ is a dominant morphism,

\item
$X'$ is reduced, and for the generic point 
$\xi' \in X'$ of any irreducible component of $X'$, we 
have $f(\xi') \not \in |D|$,  and 

\item
$f$ is a blowing up of $X$ 
along a closed subscheme of $X$.  
\end{enumerate}
\end{lem}
%
%
%
%

For a \emph{normal} variety $X$ over $k$,  
take a separably closed field $\Omega$ as $\kX \subset \Omega$. 
We denote by $\xibar$ both the geometric point $\Spec(\Omega) \to \Spec(\kX)$ 
which is given by  $\kX \subset \Omega$, 
and the geometric point $\Spec (\Omega) \to \Spec(\kX) \to X$. 
The map $\Spec(\kX) \to X$ induces a canonical and surjective homomorphism 
\begin{equation}
  \label{eq:normal}
\xymatrix@C=5mm{
  G_{\kX} = \Gal(\kXbar/\kX) \simeq \pi_1(\Spec(\kX),\xibar) \ar@{->>}[r] &\piXxi, 
}
\end{equation}
where $\ol{\kX}$ is the separable closure of $\kX$ in $\Omega$ 
(\cite{SGA1} Exp.~V, Prop.~8.2; \cite{Fu}, Prop.~3.3.6). 
Its kernel is the subgroup $\Gal(\kXbar/\kX_Z)$, 
where 
$\kX_Z$ is the maximal extension of $\kX$ unramified outside $Z$, 
that is, the subfield of $\ol{\kX}$ 
generated by all finite separable extensions $E$ of $\kX$ contained in $\ol{\kX}$ 
satisfying that  
the normalization $X^E \to X$ of $X$ in $E$ is unramified. 
In particular, the surjective map \eqref{eq:normal} 
induces an isomorphism 
\[
	\Gal(\kX_Z/\kX) \simeq  \piXxi.
\]  
In the same way, 
the fundamental group $\piXDxi$ also has a description  
\[
	\Gal(\kX_D/\kX) \simeq \piXDxi.
\] 
Here, $\kX_D$ is the subfield of  
$\kX_Z$   
generated by all finite separable extensions $E$ of $\kX$ contained in $\kX_Z$ 
satisfying that  
the normalization $X^E \to X$ is 
of ramification bounded by $D$. 
When $X$ is a smooth curve, 
it is easy to describe  
the kernel of the map $\piXxi \surj \piXDxi$ given in (\ref{eq:pi}) 
explicitly as follows:

\begin{lem}
\label{lem:str}
   Let $X$ be a smooth curve over $k$,  
   and $\Xbar$ the smooth compactification of $X$. 
   We take a geometric point $\xibar$ of $X$ and a separable closure $\kXbar$ as above.  
   For  $D = \sum m_z [z] \in\DivX$, 
   we denote by  $N$ 
   the normal closed subgroup of $\piXxi$ generated by the image of 
   the ramification subgroup  
  $G_z^{m_{z}+}$ of $G_z := \Gal(\overline{\kX_z}/\kX_z)$ 
  by the homomorphism  $G_z \to \piXxi$ 
  which is given by choosing an embedding $\ol{k(X)} \inj \ol{k(X)_z}$ 
  over $\kX$ 
  and their conjugates in $\piXxi$,  
  for all $z\in Z$. 
  Then we have
  \[
  	\xymatrix@C=5mm{
  	\piXxi/N \ar[r]^{\simeq}& \piXDxi.
  	}
  \]  
\end{lem}


For each finite \'etale morphism $f:X' \to X$ of smooth \emph{curves} over $k$, 
let $\fbar:\Xpbar\to\Xbar$ be the canonical extension of $f$ 
to the smooth compactifications. 
We denote by $D_{X'/X}$ the discriminant  
for the extension of the function fields $k(X')/k(X)$ (\cite{SerCL}, Chap.~III, Sect.~3). 
In the following, we write the discriminant additively as a divisor  
\[
  D_{X'/X} = \sum_{x\in \Xbar} v_x(D_{X'/X}) [x] 
\]
with the multiplicity $v_x(D_{X'/X}) \in \Z_{\ge 0}$ at $x$. 
The ramification locus of $\fbar$ 
coincides with the support $|D_{X'/X}| =\set{x \in \Xbar | 
v_x(D_{X'/X})>0}$ (\cite{SerCL}, Chap.~III, Sect.~5, Cor.~1)
so that 
one can consider the discriminant $D_{X'/X}$ as a Cartier divisor with support in $Z = \Xbar \ssm X$. 

\begin{prop}
\label{prop:disc}
Let $f:X'\to X$ be a finite \'etale morphism of smooth curves over $k$, 
and 
$\Xbar$ the smooth compactification. 
If we have $D_{X'/X} \le D$ for some $D\in\DivX$, 
then there exists $D'\in\DivX$ which depends only on $D$ and the degree of $f$ 
such that 
the ramification of $f:X'\to X$ is bounded by $D'$.
\end{prop}

To show this proposition, we prepare the following notation and a lemma: 
For a finite separable extension $K'/K$ of local fields, 
we denote by 

\begin{itemize}
	\item $D_{K'/K} \subset O_K$\,: the discriminant of $K'/K$, and 
	\item $\DifKp \subset O_{K'}$\,: the different of $K'/K$ (written multiplicatively as usual) 
\end{itemize}
(\Cf \cite{SerCL}, Chap.~III, Sect.~3). 

\begin{lem}
\label{lem:dif}
 	Let $K$ be a local field. 
 	Let 
	$K_1$ and $K_2$ be two finite separable extensions of $K$,  
	and $L = K_1K_2$ the compositum. 
	As ideals in the valuation ring $O_L$, we have 
	$\Dif_{L/K} \supset \Dif_{K_1/K}\Dif_{K_2/K}$.
\end{lem}
\begin{proof}
	Take $\theta \in O_{K_1}$ such that $O_{K_1} = O_K[\theta]$ 
	(\cite{SerCL}, Chap.~III, Sect.~6, Prop.~12). 
	Let $f \in O_K[T]$ be the minimal polynomial for $\theta$ over $K$. 
	From the equality $K_1 = K(\theta)$, 
	we have $L = K_1K_2 = K_2(\theta)$. 
	If we denote by $g\in O_{K_2}[T]$ the minimal polynomial for $\theta$ over $K_2$, 
	one can write $f =gh$ in $O_{K_2}[T]$ for some $h\in O_{K_2}[T]$. 
	Hence,  
	\begin{align*}
	  \Dif_{K_1/K}O_L  
	  &= f'(\theta)O_L\quad (\mbox{by \cite{SerCL}, Chap.~III, Sect.~6, Cor.~2}) \\
	  &= g'(\theta)h(\theta) O_L\quad (\mbox{by $f = gh$}) \\
	  &\subset g'(\theta)O_L \\ 
	  &\subset \Dif_{L/K_2} \quad (\mbox{by \cite{SerCL}, Chap.~III, Sect.~6, Cor.~2 again}). 
	\end{align*}
	From the transitivity of the differents (\cite{SerCL}, Chap.~III, Sect.~4, Prop.~8), 
	we obtain 
	$\Dif_{L/K} = \Dif_{L/K_2}\Dif_{K_2/K} \supset \Dif_{K_1/K}\Dif_{K_2/K}$ 
	as ideals in $O_L$.
\end{proof}

\begin{proof}[Proof of Prop.\ \ref{prop:disc}]
Let $\fbar:\Xpbar\to\Xbar$ be the canonical extension of $f$ to the smooth compactifications. 
The multiplicity of $D_{X'/X}$ at $z \in Z$ 
is written as the sum 
\[
  v_z(D_{X'/X}) = \sum_{z'\in \Xpbar,\ \fbar(z') = z} v_{z}(D_{k(X')_{z'}/k(X)_z}), 
\]
where $v_{z}(D_{k(X')_{z'}/k(X)_z})$ is 
the valuation of the discriminant 
$D_{k(X')_{z'}/k(X)_z}$ 
at $z$ (\cite{N}, Chap.~III, Sect.~2, Cor.~2.11). 
It is enough to show the following lemma. 
\end{proof}

\begin{lem} 
\label{lem:local_disc}
Let $K'/K$ be a separable 
extension of local fields with $[K':K] = n$. 
We denote by $v_K(D_{K'/K})$ the valuation of the discriminant $D_{K'/K}$. 
%
If $v_K(D_{K'/K})\le m$ for some  $m \in \Z_{\ge 0}$, 
then there exists 
$m'\in \Z_{\ge 0}$ which depends only on $m$ and $n$ 
such that 
the extension $K'/K$ is of ramification bounded by $m'$. 
\end{lem}
\begin{proof}
\textbf{(Reduction to Galois)} 
Let 
$K_1,\ldots ,K_n$ be the conjugate fields of $K' = K_1$. 
We denote by  $L = K_1\cdots K_n$ the Galois closure of $K'/K$. 
By using Lem.~\ref{lem:dif} repeatedly, we have 
$\DifL \supset \mathfrak{D}_{K_1/K}\cdots \mathfrak{D}_{K_n/K}$ 
as ideals in $O_L$. 
From the equality 
$D_{L/K} = N_{L/K}(\mathfrak{D}_{L/K})$ (\cite{SerCL}, Chap.~III, Prop.~6),   
we obtain 
\[
  v_K(D_{L/K}) \le nn!v_K(D_{K'/K}) \le nn!m. 
\]
On the other hand, 
for any $\lambda \in \R_{\ge 0}$, 
the ramification of $K'/K$ is bounded by $\lambda$ if and only if $L/K$ is of ramification bounded by $\lambda$ (Lem.\ \ref{lem:fund} (ii)). 
We may assume that $K'/K$ is a Galois extension. 


\sn 
\textbf{(Proof of the lemma)} 
It is enough to show $\Gal(K'/K)^{\lambda} = \set{1}$ for any $\lambda > m$ (Lem.\ \ref{lem:fund} (iii)). 
Take $\theta \in O_{K'}$ such that $O_{K'} = O_K[\theta]$ (\cite{SerCL}, Chap.~III, Sect.~6, Prop.~12). 
The Hilbert formula (\cite{SerCL}, Chap.~IV, Sect.~2, Prop.~4) gives 
\begin{equation}
\label{eq:Hilbert}
	v_{K'}(\DifKp) = \sum_{\sigma \in \Gal(K'/K)\ssm \set{1}} v_{K'}(\sigma(\theta) - \theta). 
\end{equation}
For any $\sigma \neq 1$ in $\Gal(K'/K)$, 
we have the following inequalities 
\begin{align*}
  m 
  & \ge v_K(D_{K'/K}) 
    \quad (\mbox{from the assumption})\\
  & \ge v_{K'}(\mathfrak{D}_{K'/K}) 
    \quad (\mbox{by $D_{K'/K} = N_{K'/K}(\mathfrak{D}_{K'/K})$})\\
  & \ge v_{K'}(\sigma(\theta)-\theta)  
    \quad  (\mbox{by \eqref{eq:Hilbert}}).
\end{align*}  
Hence, 
$\Gal(K'/K)^{\lambda} = \Gal(K'/K)_{\psi_{K'/K}(\lambda)} 
\subset \Gal(K'/K)_{\lambda}  = \set{1}$
 for any $\lambda>m$.
\end{proof}

\begin{rem}
\label{rem:disc}
For a finite \'etale morphism $f:X'\to X$ of smooth curves over $k$, 
by using the Hilbert's formula \eqref{eq:Hilbert} locally as in Prop.\ \ref{prop:disc}, 
we have the following proposition: 
If the ramification of $f:X'\to X$ is bounded by 
$D$ for some $D\in\DivX$, 
then there exists $D'\in\DivX$ which depends only on $D$ and the degree of $f$ 
such that $D_{X'/X} \le D'$.
\end{rem}

\subsection*{$l$-adic sheaves}
For a pointed connected Noetherian scheme $(X,\xbar)$ 
(in which the fixed prime $l$ is invertible \Cf Notation), 
we have an equivalence of categories (\cite{Fu}, Cor.~10.1.24):
\begin{equation}
\label{eq:l-adic}
\xymatrix@C=5mm{
\mbox{(smooth $\Qlbar$-sheaves on $X$)} \ar[r]^-{\simeq} & 
\mbox{($\Qlbar$-representations $\piXx \to \Aut(V)$)}.  
}
\end{equation} 
Here, a $\Qlbar$-representation 
we mean  
a continuous homomorphism 
$\rho:\piXx \to \Aut(V)$ with a finite dimensional $\Qlbar$-vector space $V$ 
(for the precise definition, see \cite{WII}, 1.1.6, or \cite{Fu}, Sect.~10.1).
The equivalence above is given by 
$\F \mapsto \Fx = V$, 
where $\Fx$ is the stalk of $\F$ at $\xbar$. 
\begin{lem}
 \label{lem:pullback rep}
Let $f:(X',\xpbar) \to (X,\xbar)$ be a morphism of pointed connected Noetherian schemes, 
and $\F$ a smooth $\Qlbar$-sheaf on $X$. 
Identifying the isomorphism $\Fx \simeq (f^{\ast}\F)_{\xpbar}$ on stalks,  
the $\Qlbar$-representation $\rho':\pi_1(X',\xpbar) \to \Aut((f^{\ast}\F)_{\xpbar})$ 
corresponding to $f^{\ast}\F$  
makes the following diagram commutative: 
\[
   \xymatrix@C=15mm{
     \pi_1(X,\xbar) \ar[r]^{\rho} & \Aut(\F_{\xbar}) \\
     \pi_1(X',\xpbar) \ar[r]^{\rho'}\ar[u]^{\varphi} & \Aut( (f^{\ast}\F)_{\xpbar}),\ar@{=}[u] 
   }
\]
where $\varphi$ is the induced homomorphism of fundamental groups from $f$,  
and $\rho$ is the representation corresponding to $\F$. 
\end{lem}
\begin{proof}
  For some finite extension $E$ of $\Q_l$  in $\Qlbar$ with valuation ring $R = O_E$, 
  the $\Qlbar$-sheaf $\F$ is represented by an $E$-sheaf. 
  So we may assume that $\F = (\F_n)_{n\ge 0}$ is an $\lambda$-adic sheaf, 
  where $\lambda$ is a uniformizer of $R$. 
     Recall that the representation corresponding to $\F$ is given by 
     taking the inverse limit of the representation 
     $\pi_1(X,\ol{x}) \to \Aut_{R/(\lambda^{n+1})}((\F_{n})_{\ol{x}})$
    (\Cf \cite{Fu}, proof of Prop.~10.1.23),  
  and  $f^{\ast}\F = (f^{\ast}\F_n)_{n\ge 0}$ by the very definition. 
  Without loss of generality, we may assume that 
  $\F$ is a locally constant sheaf on $X$ with finite stalks 
  and $\rho:\piXx \to \Aut(\Fx)$ 
  is given by the action of $\piXx$ on $\Fx$. 
  There exists an \'etale covering $Y \to X$ such that 
  $\ul{Y} \simeq \F$, 
  where $\ul{Y} := \Hom_X(-,Y)$ (\cite{Fu}, Prop.~5.8.1 (i)). 
  On the other hand, 
  we have $f^{\ast}\F\simeq \ul{Y\times_X X'}$ (\cite{Fu}, Prop.~5.2.7). 
  The fiber functors at $\xbar$ and $\xpbar$ (\ref{eq:Gal}) make the following diagram commutative: 
  \[
    \xymatrix@C=15mm{
      \Cov(X)\ar[r]^-{\simeq}\ar[d]_{-\times_X X'} 
      & \piXxSets\ar[d]^{\varphi_{\ast}} \\
      \Cov(X')\ar[r]^-{\simeq} & \piXpxSets,
    }
  \]
  where $\varphi_{\ast}$ is the functor induced by the homomorphism  $\varphi$ 
  (\cite{Fu}, Prop.~3.3.1). 
  As a result, 
  the action of $\pi_1(X',\xpbar)$ on $(f^{\ast}\F)_{\xpbar}$ 
  comes from that of $\piXx$ on $\F_{\xbar}$. 
\end{proof}
%
%

\begin{dfn}[\cite{EK}, Sect.~3, see also \cite{KR}, Sect.~10.1]
\label{def:Swan}
\begin{enumerate}
\item 
Let 
$K$ be a local field with residue field of characteristic $p$. 
For a $\Qlbar$-representation $\rho:G_K \to \Aut(V)$,  
the \textbf{Swan conductor} of $V$ is defined by  
\[
	\SwV = \displaystyle{\sum_{\lambda>0}}\lambda \dim(V^{G_K^{\lambda+}}/V^{G_K^{\lambda}}), 
\] 
where $V^{G_K^{\lambda+}}$ and $V^{G_K^{\lambda}}$ are 
the fixed subspace of $V$ by $\rho(G_K^{\lambda+})$ and $\rho(G_K^{\lambda})$ 
respectively.  
 
\item 
Let $X$ be a smooth \emph{curve} over $k$, 
and $\Xbar$ the smooth compactification of $X$. 
For 
a smooth $\Qlbar$-sheaf $\F$ on $X$, 
the \textbf{Swan conductor} of $\F$ is defined to be the effective 
Cartier divisor 
\[
	\SwF = \sum_{z \in Z} \SwzF [z] \in \DivX.
\]
Here, for each $z\in Z$, 
$\SwzF := \Sw(\F|_{\Spec(\kXz)}) \in \Z_{\ge 0}$ (Lem.\ \ref{lem:oplus} (i) below) is the Swan conductor of 
(the $\Qlbar$-representation corresponding to)
the restriction $\F|_{\Spec(\kXz)}$.

\item
Let $X$ be a \emph{variety} over $k$.
For $D \in \DivX$  
and for a smooth $\Qlbar$-sheaf $\F$ on $X$,  
we say that the \textbf{ramification of $\F$ is bounded by $D$} 
(and write as $\SwF \le D$ formally) 
if, for every $\phi:C\to X \in \CuX$, we have 
$\Sw(\phiastF) \le \phibarastD$,
where $\phiastF$ is  
the pullback of $\F$ by $\phi$. 
\end{enumerate}
\end{dfn}

\begin{lem}[\cite{KR}, Thm.~4.85,  \cite{EK}, Sect.~3.1, (3.1)]
\label{lem:oplus}
Let $K$ be a local field with residue field of characteristic $p$. 

\begin{enumerate}
	\item For a $\Qlbar$-representation $V$ of $G_K$, 
	the Swan conductor $\Sw(V)$ takes a value in $\Z_{\ge 0}$. 
	\item 
	For two $\Qlbar$-representations $V$ and $V'$ of $G_K$,  
	we have $\Sw(V\oplus V' ) = \Sw(V) + \Sw(V')$.
	\end{enumerate}
\end{lem}

\begin{rem}
\label{rem:Ar}
For use later (in Lem.\ \ref{lem:piab}), we refer to the Artin conductor 
(\Cf \cite{SerCL}, Chap.~IV and \cite{U}, Sect.~4). 

\begin{enumerate}
	\item 
	Let $K$ be a local field with residue field of characteristic $p$, and 
$\rho:G_K \to \Aut(V)$ a $\Qlbar$-representation of $G_K$. 
We assume that the restriction $\rho|_{G^0_K}$ of $\rho$ 
to $G^0_K$ has finite image. 
%
The \textbf{Artin conductor} of $V$ is defined by 
\[
  \Ar(V) = \int_{-1}^{\infty} \dim (V/V^{G_K^x}) dx. 
\]
Putting 
\[
  \epsilon(V) = \int_{-1}^0 \dim (V/V^{G_K^x}) dx = \dim(V/V^{G_K^0}),   
\]
we have  
\begin{align*}
  \Ar(V) 
  &= \epsilon(V) + \int_{0}^{\infty}\dim(V/V^{G_K^{x}})dx\\
  & = \epsilon(V) + \Sw(V).
\end{align*}
%
\item 
Let $\F$ be a smooth $\Qlbar$-sheaf on a smooth curve $X$ over $k$. 
We assume that the corresponding $\Qlbar$-representation 
$\rho:\piXx \to \Aut(\Fx)$ by \eqref{eq:l-adic} has finite image, for simplicity.
The global conductors of $\F$ are defined by  
\[
	\Ar(\F) = \sum_{z \in Z} \Ar_z(\F) [z]\quad\mbox{and}\quad \epsilon(\F) = \sum_{z \in Z}\epsilon_z(\F)  [z] \in \DivX, 
\]
by using 
the local conductors $\Ar_z(\F) := \Ar(\F|_{\Spec(\kXz)})$ and 
$\epsilon_z(\F) := \epsilon(\F|_{\Spec(\kXz)})$ respectively. 
From the relation of the local Artin and Swan conductors noted above, we have 
\begin{equation}
	\label{eq:epsilon}
	\Ar(\F) = \epsilon(\F) + \Sw(\F).
\end{equation}
\end{enumerate}
\end{rem}

\begin{lem}
  \label{lem:bdd}
  Let $(X,\xbar)$ be a pointed variety over $k$, and 
  $\F$ a smooth $\Qlbar$-sheaf on $X$ of rank $r$. 
  If the corresponding $\Qlbar$-representation $\rho:\piXx \to \Aut(\Fx)$ 
  factors through $\piXDx$ for some $D \in \DivX$, 
  then we have 
  $\SwF \le rD$.  
\end{lem}
\begin{proof}
  For each $\phi:C\to X \in \CuX$, we show  
  $\Sw(\phiastF) \le \phibarast(rD) = r(\phibarastD)$ 
  (the last equality follows from the additivity property of $\phibarast$: 
  $\phibarast(D + D') = \phibarastD + \phibarast D'$, \Cf \cite{EGA}, Sect.~21.4.2).  
  The assumption on $\rho$ 
  does not depend on the choice of the geometric point $\xbar$ of $X$. 
  By replacing the geometric point $\xbar$ if necessary, 
  the morphism $\phi:C\to X \in \CuX$ induces a commutative diagram 
  \begin{equation}
  \label{eq:phi}
  \vcenter{
	  \xymatrix@C=15mm{
  	    \piCa\ar[r]^{\varphi} \ar@{->>}[d] & \piXx \ar@{->>}[d] \\
      	\piCDa \ar[r] & \piXDx
      }
  }
  \end{equation}
  by Lem.~\ref{lem:pull-back}, 
  where $\varphi$ is the induced homomorphism of fundamental groups from $\phi$, 
  and $\abar$ a geometric point of $C$. 
  By Lem.~\ref{lem:pullback rep}, 
  the pullback $\phiastF$ 
  corresponds to the representation given by the composition 
  $\piCa \onto{\varphi} \piXx \onto{\rho}  \Aut(\Fx)$.
  From the assumption and the diagram (\ref{eq:phi}), this representation factors 
  through $\piCDa$. 
  From this, we may assume that $X$ is a smooth curve over $k$. 
  
   Since the Swan conductor is defined to be 
  $\SwF = \sum_{z \in Z} \SwzF [z]$, 
  if we write $D = \sum_{z \in Z} m_z [z]$, 
  it is enough to show $\SwzF \le r m_z$ 
  for each $z\in Z$. 
  By replacing the geometric point and using Lem.~\ref{lem:pullback rep} again, 
  the multiplicity $\SwzF$ at $z\in Z$  
  is given by the Swan conductor of the representation 
  $G_z \to \piXx \onto{\rho}  \Aut(\Fx)$ 
  of $G_z :=\Gal(\kXzbar/\kXz)$,  
  where the first homomorphism 
  is given 
  by choosing an embedding $\kXbar \inj \kXzbar$ over $\kX$ as in Lem.\ \ref{lem:str}. 
  The assumption and Lem.~\ref{lem:str} imply that this representation factors 
  through $G_z/G_z^{m_z+}$. 
  The assertion is reduced to showing the lemma below.    
\end{proof}

\begin{lem}
  Let $K$ be a local field with residue field of characteristic $p$. 
  If a $\Qlbar$-representation
  $\rho:G_K \to \Aut(V)$ 
  annihilates $G_K^{\lambda+}$ for some $\lambda \in \R_{\ge 0}$, 
  then we have $\Sw(V) \le \dim(V) \lambda$. 
\end{lem}
\begin{proof}
  Let $L = \Kbar^{\Ker(\rho)}$ be the extension of $K$  
  corresponding to $\Ker(\rho)$. 
  From the assumption, we have $G_K^{\lambda+} \subset \Ker(\rho) = G_L$. 
  The ramification of $L/K$ is bounded by $\lambda$. 
  In particular, we have $V^{G_K^{\lambda'}} = V$ 
  for any $\lambda' > \lambda$. 
  We obtain 
  \[
    \Sw(V) = \sum_{0<\lambda' \le \lambda }\lambda' 
    \dim(V^{G_K^{\lambda'+}}/V^{G_K^{\lambda'}}) \le \lambda\dim(V). 
  \]
\end{proof}


\section{Finiteness}
\label{sec:small}

\subsection*{Smallness}
We recall a property of profinite groups 
called \textit{smallness}. 
This notion is also refereed as ``type (F)'' in \cite{SerCG}, Chap.~III, Sect.~4.1.

\begin{dfn}[\cite{HH}, Def.~2.1]
\label{def:small}
  A profinite group $G$ is said to be \textbf{small} 
  if there exist only finitely many open subgroups $H$
  with $(G:H) \le n$ for any $n\in \Z_{\ge 1}$.
\end{dfn}

For example, 
a topologically finitely generated profinite group 
is small (\cite{HH}, Prop.~2.4). 
Using this notion, 
one can interpret the Hermite-Minkowski theorem 
as follows:  
For a number field $F$ and 
a finite set $S$ of primes of  $F$, 
the Galois group $G_{S}$ of the maximal Galois extension
of  $F$ unramified outside $S$ is small.

\begin{prop}[\cite{HH}, Sect.~2] 
\label{prop:small}
Let $G$ and $G'$ be profinite groups. 
\begin{enumerate}
\item
  $G$ is small if and only if 
  there exist only finitely many open normal subgroups $N$
  with $(G:N) \le n$ for any $n\in\Z_{\ge 1}$. 

\item
  If $G$ is small and 
  $N$ is a closed normal subgroup of $G$,
  then the quotient group $G/N$ is small.
  
\item
  If $G$  and $G'$ are small, then 
  their free product $G\bigast G'$ is also small. 
\end{enumerate}
\end{prop}

%

\subsection*{Main theorem}

Recall that 
a {smooth Weil sheaf} $\F$ on a variety $X$ over $k$   
consists of a smooth $\Qlbar$-sheaf $\ol{\F}$ on 
$X\otimes_k \kbar:= X\times_{\Spec(k)}\Spec(\kbar)$  
and an action of the Weil group $W_k=W(\kbar/k) \subset G_k$ on $\ol{\F}$ (\cite{WII}, Def.~1.1.10 (i)).  
For a smooth $\Qlbar$-sheaf $\F$ on $X$, 
the pullback of $\F$ on $X\otimes_k \kbar$ 
by the projection $X\otimes_k \kbar \to X$ 
and the action of $W_k$ which is given by the restriction of that of $G_k$  
produce a smooth Weil sheaf on $X$. 
This construction gives a fully faithful functor
\begin{equation}
\label{eq:Weil} 
\xymatrix@C=5mm{
\mbox{(smooth $\Qlbar$-sheaves on $X$)}\ \ar@{^{(}->}[r] & \mbox{(smooth Weil sheaves on $X$)}.
}
\end{equation}
The Weil sheaves in the essential image of the above functor  
are said to be \textbf{\'etale} (\cite{WII}, 1.3.2). 
In fact, 
for a general Weil sheaf $\F$ on $X$ 
and $D\in \DivX$ 
the condition ``$\Sw(\F) \le D$'' 
is defined in the same manner as Def.~\ref{def:Swan} 
using the Weil group $\WXx \subset \piXx$ 
(\cite{EK}, Sect.~3).  
For a smooth $\Qlbar$-sheaf $\F$, 
we have 
\begin{equation}
\label{eq:Sws}
\xymatrix@C=7mm{
  \Sw(\F) \le D \mbox{ in the sense of Def.~\ref{def:Swan}} 
  \ar@{<=>}[r]& \Sw(\F \mbox{ as an \'etale Weil sheaf}\, ) \le D.  
}  
\end{equation}
From this, we often identify smooth $\Qlbar$-sheaves with 
the corresponding \'etale Weil sheaves. 
Adding to \cite{WII}, 
for more details on Weil sheaves, 
see also \cite{KW}, Chap.~I, and \cite{KR}, Sect.~10. 
Following \cite{EK}, for $r\in \Z_{\ge 1}$, we denote by

\begin{itemize} 
\item $\RrX$\,: the set of smooth Weil sheaves on $X$ 
of rank $r$ up to isomorphism and up to semi-simplification. 
\end{itemize}

\begin{thm}[\cite{EK}, Thm.~2.1]
\label{thm:Deligne}
Let $X$ be a smooth variety over $k$, and 
$\Xbar$ a normal compactification of $X$ such that 
$Z = \Xbar \ssm X$ is the support of an effective Cartier divisor on $\Xbar$.  
Then for any $r, n \in\Z_{\ge 1}$ and $D \in \DivX$, 
the set of irreducible sheaves $\F \in \RrX$   
with 

\begin{itemize}
\item $\SwF \le D$, and 
\item $\det(\F)^{\otimes n} = 1$ 
\end{itemize}
is finite, 
where $\det(\F) := \bigwedge^r \F$ is the determinant of $\F$. 
\end{thm}

From the Galois correspondence within the Galois category $\Cov(X,D)$ 
for a variety $X$ over $k$, 
Thm.~\ref{thm:intro} is equivalent to the smallness of 
the fundamental group $\piXDx$ which is stated in Thm.~\ref{thm:main} below. 
For another geometric point $\xpbar$ of $X$,  
we have an isomorphism $\pi_1(X,D, \xpbar) \isomto  \piXDx$ 
and two such isomorphisms differ by an inner automorphism of 
$\piXDx$ 
(\cite{Fu}, Prop.~3.2.13). 
Since we are interested in 
the smallness, 
we omit the base point $\xbar$ of $X$ 
from the fundamental groups and write $\piXD$ as well as $\piX$. 
Under this convention, 
it must be noted that a \emph{canonical}\/ homomorphism
or a \emph{commutative}\/ diagram 
concerning these groups 
makes sense viewing it up to conjugates.

\begin{thm}
\label{thm:main}
  Let $X$ be a variety over a finite field $k$ of characteristic $p$,  
  $\Xbar$ a compactification of $X$, 
  and $Z = \Xbar \ssm X$.  
  Then $\piXD$ is small for any $D\in \DivX$. 
\end{thm}
\begin{proof}
First, we reduce the assertion to the situation in which 
Thm.~\ref{thm:Deligne} holds, that is,  
$X$ is smooth over $\Spec(k)$ and the compactification $\Xbar$ is normal 
with the boundary $Z = \Xbar\ssm X$ 
which is the support of an effective Cartier divisor on 
$\Xbar$. 

\sn
\textbf{(Reduction to reduced)}
Let $\Xbarred$ and  $\Xred$ be the reduced closed 
subschemes associated to $\Xbar$ and $X$ respectively. 
The natural morphisms 
$f:\Xred \inj X$ and $\fbar:\Xbarred \inj \Xbar$ 
which are closed immersions 
give the following commutative diagram 
\[
	\xymatrix@C=15mm{
	\Xbarred\ar@{^{(}->}[r]^{\fbar} & \Xbar \\
	\Xred \ar@{^{(}->}[u]^{j_{\red}}\ar@{^{(}->}[r]^f & X, \ar@{^{(}->}[u]_j 
	}
\]
where $j_{\red}$ is the morphism associated to 
the inclusion map $j$ 
(\cite{EGA}, Chap.~I, 5.1.5). 
Since we have $|\Xbarred| = |\Xbar|$ and $|\Xred| =|X|$ 
as the underlying topological spaces, 
the induced morphism $j_{\red}$ is a dominant open immersion 
(\cite{EGA}, Chap.~I, Prop.~5.1.6). 
Here, $\Xbarred$ is reduced, 
and for the generic point $\xi$ of an irreducible component of $\Xbarred$, 
we have $\fbar(\xi) \not \in Z = \Xbar\ssm X$.
In particular, $\fbar(\xi) \not \in |D|$. 
The inverse image $\fbarastD \in \Div_{\Xbarred \ssm \Xred}(\Xbarred)$ 
is defined (the case (c) in Lem.~\ref{lem:divisor}). 
By Lem.~\ref{lem:pull-back}, there is a commutative diagram 
\begin{equation}
\label{eq:pired}
\vcenter{
	\xymatrix@C=15mm{
	\pi_1(X_{\red})\ar@{->>}[d]\ar[r] & \piX\ar@{->>}[d] \\
	\pi_1(\Xred, \fbarastD) \ar[r] & \piXD, 
	}
}
\end{equation}
where the vertical homomorphisms are surjective 
as we noted in (\ref{eq:pi}).  
Since the top horizontal homomorphism in (\ref{eq:pired}) 
is known to be bijective (\cite{SGA1}, Exp.~IX, Prop.~1.7), 
the bottom is surjective. 
If we assume that $\pi_1(\Xred,\fbarastD)$ is small,  
then so is $\piXD$ by Prop.~\ref{prop:small} (ii).  
Therefore, we may assume that $X$ and $\Xbar$ are reduced. 

\sn
\textbf{(Reduction to integral)}
Since $\Xbar$ is Noetherian, it has only a finite number of 
irreducible components $\Xbar_1,\ldots ,\Xbar_n$. 
The irreducible components of $X$ are given by 
$X_i := \Xbar_i\cap X$ for $i=1,\ldots ,n$ (\cite{Liu}, Prop.~2.4.5(b)). 
We endow these components with the reduced closed subscheme structure.
For each $i$, we obtain the following commutative diagram 
\[
	\xymatrix@C=15mm{
	\Xbar_i\ar@{^{(}->}[r]^{\fibar} & \Xbar \\
	X_i \ar@{^{(}->}[u]\ar@{^{(}->}[r]^{f_i} & X,\ar@{^{(}->}[u]
	}
\]
where $\fibar$ is the natural morphism 
and $f_i$ is induced by $\fibar$. 
The inverse image 
$D_i := \fibarast D \in \Div_{\Xbar_i\ssm X_i}(\Xbar_i)$ exists. 
In fact, for the generic point $\xi_i \in \Xbar_i$, we have $\fibar(\xi_i) 
\not\in |D|$ (the case (c) in Lem.~\ref{lem:divisor}). 
The canonical homomorphism $\pi_1(X_i,D_i) \to \piXD$ exists for each $i$ by Lem.~\ref{lem:pull-back}.
The collection of morphisms $\set{f_i}_{1\le i \le n}$ induces 
$\bigsqcup_{i=1}^n X_i \to X$ which is an 
effective descent morphism (\cite{SGA1}, Exp.~IX, see also \cite{S}, Thm.~5.2). 
The descent theory (\cite{SGA1}, Exp.~IX, Thm.~5.1, see also \cite{S}, Cor.~5.3) 
says that there exist finitely many generators 
$\gamma_1,\ldots , \gamma_m$ 
such that 
we have a surjective homomorphism 
\[
\xymatrix@R=5mm@C=5mm{
 \displaystyle{\Conv_{i=1}^n \pi_1(X_i) \Asterisk \braket{\gamma_1,\ldots ,\gamma_m}}   
 \ar@{->>}[r] & \piX, 
}
\]
where 
$\braket{\gamma_1,\ldots, \gamma_m}$ is the profinite completion of the free group 
on the set $\set{\gamma_1,\ldots,\gamma_m}$. 
The fundamental groups 
$\pi_1(X_i,D_i)$ and $\piXD$ are 
quotients of $\pi_1(X_i)$ and $\piX$ respectively as in (\ref{eq:pi}). 
The canonical homomorphisms $\pi_1(X_i,D_i) \to \piXD$ 
and the composition $\braket{\gamma_1,\ldots ,\gamma_m} \to \piX \surj \piXD$ 
give a commutative diagram 
\begin{equation}
\label{eq:piint}
\vcenter{
\xymatrix@R=5mm@C=15mm{
 \displaystyle{\Conv_{i=1}^n \pi_1(X_i) \bigast \braket{\gamma_1,\ldots ,\gamma_m}}\ar@{->>}[d]   
 \ar@{->>}[r] & \piX \ar@{->>}[d] \\
  \displaystyle{\Conv_{i=1}^n \pi_1(X_i,D_i) \bigast  \braket{\gamma_1,\ldots ,\gamma_m}}   
  \ar[r] & \piXD.
 }
}
\end{equation}
From the diagram (\ref{eq:piint}), 
the bottom horizontal map is surjective.  
As the free profinite group 
$\braket{\gamma_1,\ldots,\gamma_m}$ is topologically finitely generated, 
it is small (\cite{HH}, Prop.~2.4).  
If we assume $\pi_1(X_i, D_i)$ is small for all $i$, 
the free product 
$\bigast_{i=1}^n \pi_1(X_i,D_i) \bigast  \braket{\gamma_1,\ldots ,\gamma_m}$
is small by Prop.~\ref{prop:small} (iii) 
and the same holds for $\piXD$ by Prop.~\ref{prop:small} (ii). 
Thus, we may assume that $X$ and $\Xbar$ are integral. 

\sn
\textbf{(Reduction to normal)}
Let $\fbar:\Xbar' \to \Xbar$ be the normalization of $\Xbar$. 
The morphism $\fbar$ is finite (\cite{Liu}, Cor.~4.1.30) 
and hence proper. 
The restriction  
$f:X' := \fbar^{\,-1}(X) \to X$ of $\fbar$ 
gives 
\[
\xymatrix@C=15mm{
   \Xpbar \ar[r]^{\fbar} & \Xbar \\
   X' \ar@{^{(}->}[u]\ar[r]^f & X \ar@{^{(}->}[u]
}
\]
which is commutative. 
The morphism $f$ is also the normalization of $X$ 
(from the very definition of the normalization, \cite{Liu}, Cor.~4.1.19).  
By the descent theory again  (\cite{SGA1}, Exp.~IX, Thm.~5.1, see also \cite{S}, Cor.~5.3), 
$\piX$ is a quotient of the free product of $\pi_1(X')$ and 
a free profinite group $\braket{\gamma_1,\ldots,\gamma_m}$ as  
\[
\xymatrix@R=5mm@C=5mm{
 \displaystyle{\pi_1(X') \Asterisk \braket{\gamma_1,\ldots , \gamma_m}}   
 \ar@{->>}[r] & \piX. 
}
\]
The inverse image $D' := \fbarast(D)$ on $\Xpbar$ 
exists by Lem.~\ref{lem:divisor} 
(the case (b)). 
From Lem.~\ref{lem:pull-back} we have the following commutative diagram:  
\begin{equation}
\label{eq:pinormal}
\vcenter{
\xymatrix@R=5mm@C=15mm{
 \displaystyle{\pi_1(X') \bigast \braket{\gamma_1,\ldots , \gamma_m}}\ar@{->>}[d]   
 \ar@{->>}[r] & \piX \ar@{->>}[d] \\
  \displaystyle{\pi_1(X',D') \bigast\braket{\gamma_1,\ldots , \gamma_m}}   
   \ar[r] & \piXD.
 }
 }
\end{equation} 
The bottom homomorphism in the diagram (\ref{eq:pinormal}) becomes surjective. 
By Prop.~\ref{prop:small} (ii) and (iii)  
as above (the reduction to integral), 
the assertion is reduced to the situation where 
$X$ and $\Xbar$ are normal. 

\sn
\textbf{(Reduction to smooth)}
Take the smooth locus $U$ of $X$. 
This $U$ forms an open subscheme of $X$ (\cite{Liu}, Prop.~8.2.40).  
Since $X$ and hence $U$ are normal, 
$\piX$ and $\piU$ are quotient of the absolute Galois group 
$G_{\kX}$ and $G_{\kU}$ respectively \eqref{eq:normal}. 
As these quotient maps are canonical (up to inner automorphisms), 
we have the following commutative diagram: 
\[ 
\xymatrix@C=28mm{
  G_{k(U)} \ar@{->>}[d]\ar@{=}[r] &  G_{\kX}\ar@{->>}[d]\\
  \pi_1(U)\ar[r] & \piX.
}
\]
By the commutativity, the bottom horizontal map is surjective. 
From Lem.~\ref{lem:pull-back}, 
there is a commutative diagram 
\[ 
\vcenter{
\xymatrix@C=15mm{
  \pi_1(U) \ar@{->>}[r]\ar@{->>}[d] & \piX \ar@{->>}[d]\\
  \pi_1(U\subset \Xbar,D)\ar[r] & \pi_1(X\subset \Xbar,D).
}
}
\]
Since the top horizontal homomorphism 
$\pi_1(U) \to \piX$ in the diagram above is surjective,  
so is the bottom. 
From Lem.~\ref{prop:small} (ii), 
we may assume that $X$ is smooth over $k$. 

\sn
\textbf{(Reduction to $Z = |E|$ for some $E\in \DivX$)}
Let $\pbar:\Xbar' \to \Xbar$ be the blowing up of $\Xbar$ along 
the reduced closed subscheme $Z = \Xbar \ssm X \inj \Xbar$. 
Note that $\Xbar'$ is integral (\cite{EGA}, Chap.~II, Prop.~8.1.4) 
(but may not be normal), 
$\pbar$ is proper and 
induces an isomorphism $X' := \pbar^{-1}(X) \isomto X$ (\cite{Liu}, Prop.~8.1.12 (b) and (d)) 
which makes the following diagram commutative: 
\[
	\xymatrix@C=15mm{
	\Xbar'\ar[r]^{\pbar} & \Xbar \\
	X' \ar@{^{(}->}[u]\ar[r]^-{\simeq} & X.\ar@{^{(}->}[u]
	}
\]
Hence, we have a dominant open immersion $X' \inj \Xbar'$ 
whose complement $\Xbar' \ssm X'$ is 
the support of an effective Cartier divisor on $X'$ 
(\cite{Liu}, Prop. 8.1.12 (e)).
The inverse image $D' := \pbar^{\ast}D$ exists 
(the case (d) in Lem.~\ref{lem:divisor}).  
Lem.~\ref{lem:pull-back} 
gives a commutative diagram 
\begin{equation}
	\label{eq:blowup}
\vcenter{
\xymatrix@C=15mm{
  \pi_1(X')\ar@{->>}[d] \ar[r]^{\simeq}   & \piX\ar@{->>}[d]\\
  \pi_1(X'\subset \Xpbar,D') \ar[r] & \pi_1(X\subset \Xbar, D). 
} }
\end{equation}
From this diagram, 
the bottom horizontal map 
$\pi_1(X' \subset \Xbar',D' ) \surj \pi_1(X\subset \Xbar ,D)$ 
is surjective.
Using Prop.~\ref{prop:small} (ii), we may assume that 
$Z = \Xbar \ssm X$ is the support of 
an effective Cartier divisor $E$ on $\Xbar$.

\sn
\textbf{(Reduction to a normal compactification)}
Let $\fbar:\Xbar' \to \Xbar$ be the normalization.  
This morphism $\fbar$ is finite and 
induces an isomorphism $X':= {\fbar}^{\,-1}(X) \isomto X$ 
as $X$ is normal. 
From Lem.~\ref{lem:pull-back}, the diagram 
\[
\xymatrix@C=15mm{
   \Xpbar \ar[r]^{\fbar} & \Xbar \\
   X' \ar@{^{(}->}[u]\ar[r]^{\simeq} & X \ar@{^{(}->}[u]
}
\]
induces a commutative diagram 
\begin{equation}
	\label{eq:normal_comp}
	\vcenter{
\xymatrix@C=15mm{
  \pi_1(X')\ar@{->>}[d] \ar[r]^{\simeq}   & \piX\ar@{->>}[d]\\
  \pi_1(X'\subset \Xpbar,D') \ar[r] & \pi_1(X\subset \Xbar, D), 
} }
\end{equation}
where $D' = {\fbar}^{\ast}(D)$. 
In particular, we obtain a surjective homomorphism 
$\pi_1(X'\subset\Xbar',D' ) \surj  \pi_1(X\subset \Xbar,D)$.
Accordingly, there exists an open immersion $X' \inj \Xbar'$ 
with $|\Xbar'\ssm X'| = |\fbarast E|$.
Therefore, without loss of generality, 
we may assume that $X$ is smooth and 
$\Xbar$ is a 
normal compactification of $X$ whose boundary 
$Z = \Xbar\ssm X$ is the support of an effective Cartier divisor on $\Xbar$. 

Next, we show the smallness of the fundamental group $\piXD$ 
under these assumptions. 

\sn
\textbf{(Proof of the smallness)} 
  For $r\in \Z_{\ge 1}$, we denote by $\Sr$  
  the set of  open \emph{normal} subgroups $N$ of $\piXD$ with 
  $(\piXD:N) =r$. 
  By Prop.~\ref{prop:small} (i), 
  it is enough to show $\#\Sr < \infty$. 
  For each $N\in \Sr$, 
  consider the regular (and semi-simple) representation 
  $\rhobar_N:\piXD/N \inj \GL_r(\Qlbar)$
  of the finite group $\piXD/N$ (with $l\neq p$).  
  By composing the natural homomorphisms (\Cf (\ref{eq:pi})), 
 $\rhobar_N$ induces 
  \[
  \xymatrix@C=5mm{
 \rho_N:\piX \ar@{->>}[r] & \piXD \ar@{->>}[r] & \piXD/N\, \ar@{^{(}->}[r]^-{\rhobar_N} & \GL_r(\Qlbar) .
  }
  \]
  We denote by $\F_N \in \RrX$ 
  the \'etale Weil sheaf on $X$ associated with $\rho_N$ (\Cf (\ref{eq:l-adic}) and (\ref{eq:Weil})).
  By considering the irreducible decomposition, 
  $\F_N$ consists of irreducible components $\F$ of rank $\le r$ 
  satisfying 
  
  \begin{itemize}
  \item $\Sw(\F) \le \Sw(\F_N) \le rD$ \quad (by Lem.~\ref{lem:oplus} (ii), Lem.~\ref{lem:bdd} and (\ref{eq:Sws})), and 
  
  \item $\det(\F)^{\otimes r} = \det(\F_N)^{\otimes r} = 1$ \quad 
  (for  
  $(\piXD:N) = r$, and hence $\det(\rho_N)^{\otimes r} = 1$).
  \end{itemize}
  
  \noindent
  Thm.~\ref{thm:Deligne} implies that 
  there exist only finitely many such irreducible sheaves. 
  For we have $\#\Sr = \#\set{\F_N}_{N\in \Sr}$,    
  the assertion $\#\Sr < \infty$ follows.
\end{proof}

\subsection*{Applications}
Let $F$ be an algebraically closed field, 
$X$ a  variety over $k$, and 
$D\in \DivX$. 
As in Sect.~4 of \cite{HH}, in the following, 
we derive some finiteness results of representations 
$\piXD \to\GL_r(F)$ 
from our main theorem (Thm.~\ref{thm:main}) 
(so that we still omit the base point).  
Here, we endow $\GL_r(F)$ with the discrete topology.  
We define 
\begin{equation}
	\label{eq:pi0}
\xymatrix@C=5mm{
  \piXD^0 := \Ker(\varphi:\piXD \ar[r]& \pi_1(\Spec(k))= G_k),
}
\end{equation}
where the homomorphism $\varphi$ is induced from the structure morphism $X\to \Spec(k)$. 
Since $G_k (\simeq \Zhat)$ is abelian, $\piXD^0$ does not depend on the choice of 
the (omitted) geometric point.

\begin{dfn}
A representation $\rho: \piXD\to\GL_r(F)$ is said to be \textbf{geometric}
if it satisfies $\rho(\piXD) = \rho(\piXD^0)$. 
\end{dfn}
When $X$ is normal, 
this is equivalent to 
that the corresponding extension of the function field $k(X)$  
contains no constant field extension. 

\begin{cor}
\label{cor:finite}
Let $X$ be a normal variety over $k$,  
$\Xbar$ a compactification of $X$, and 
$Z = \Xbar \ssm X$. 
For $r\in \Z_{\ge 1}$ and $D\in \DivX$, we have the following: 

\begin{enumerate}
\item
There exist only finitely many isomorphism classes of semi-simple geometric 
representations 
$\piXD \to \GL_r(F)$ with solvable image. 
 
\item
If the characteristic of $F$ is $0$, then 
there exist only finitely many isomorphism classes of semi-simple geometric 
representations 
$\piXD \to \GL_r(F)$. 
\end{enumerate}
\end{cor}

To show this corollary, we prepare some notation and a lemma on a finiteness of the abelian fundamental group. 
For a topological group $G$, 
we denote by 

\begin{itemize}
	\item $G^{\vee} = \Set{\chi \in \Hom(G,\QZ)| \mbox{continuous}}$\,: the Pontrjagin dual group of $G$, and 
	\item $G^{\ab} = G/\ol{[G,G]}$\,: the abelianization of $G$, where $\ol{[G,G]}$ 
	 is the topological closure of the commutator subgroup of $G$. 
\end{itemize}
The structure map $X\to \Spec(k)$ as in \eqref{eq:pi0} induces 
\[
\xymatrix@C=5mm{
  \piXDabg := \Ker(\varphi^{\ab}:\piXDab \ar[r]& \pi_1(\Spec(k))^{\ab} = G_k). 
}
\]

\begin{lem}
	\label{lem:piab}
  Let $X$ be a normal variety over $k$, and 
  $\Xbar$ a compactification of $X$. 
  For any $D\in \DivX$, 
  we have $\# \piXDabg < \infty$.   
\end{lem}
\begin{proof}
  As in the proof of the main theorem (Thm.\ \ref{thm:main}), 
  we reduce the lemma to the case where 
  $\Xbar$ is normal and the boundary 
  $Z = \Xbar\ssm X$ is the support of an effective Cartier divisor on $\Xbar$. 
	
  \sn 
  (\textbf{Reduction to $Z=|E|$ for some $E\in \DivX$)}\, 
  Let $\pbar:\Xpbar \to \Xbar$ be the blowing up  
  along the reduced closed subscheme $Z = \Xbar \ssm X$. 
  As in \eqref{eq:blowup}, 
  there exists a commutative diagram 
  \[
	  \xymatrix@C=15mm{
	    \pi_1(X')^{\ab} \ar@{->>}[d]\ar[r]^{\simeq} & \piXab\ar@{->>}[d]\\
	    \pi_1(X'\subset \Xpbar,{\pbar}^{\ast}D)^{\ab} \ar[r] & \pi_1(X\subset \Xbar,D)^{\ab} 
	  }
  \]
  with $X' = (\pbar)^{-1}(X)$, 
  and the bottom horizontal map becomes surjective.  
  We may assume that $Z = \Xbar \ssm X$ is the support of an effective Cartier divisor on $\Xbar$. 

\sn	
   (\textbf{Reduction to a normal compactification)}\, 
	Let $\fbar:\Xpbar \to \Xbar$ be the normalization. 
	As in \eqref{eq:normal_comp}, 
	the isomorphism $X' := {\fbar}^{-1}(X) \isomto X$ induces 
	a commutative diagram 
	\[
	  \xymatrix@C=15mm{
	    \pi_1(X')^{\ab} \ar@{->>}[d]\ar[r]^{\simeq} & \piXab\ar@{->>}[d]\\
	    \pi_1(X'\subset \Xpbar,{\fbar}^{\ast}D)^{\ab} \ar[r] & \pi_1(X\subset \Xbar,D)^{\ab}.  
	  }
	\]
	The bottom map is surjective so that 
	we may assume that $\Xbar$ is normal 
	and there exists $E\in \DivX$ such that $Z = |E|$. 
	 
	\sn
	(\textbf{Proof of the lemma}) 
	 Recall that there exists a canonical isomorphism 
$(\piXab)^{\vee} \simeq \HX$, where $\HX$ denotes the \'etale cohomology group (\cite{SGA4.5} Exp.~1, Sect.~2.2.1). 
By composing this with the dual of $\piXab \surj \piXDab$  induced from \eqref{eq:pi}, 
we have an injective homomorphism 
\begin{equation}
\label{eq:H1}
\xymatrix@C=5mm{
  \psi: (\piXDab)^{\vee} \ar@{^{(}->}[r] & \HX. 
}
\end{equation}
Following \cite{KS}, Def.~2.4, 
we define 
\[
   \filDHX := \Set{\chi \in \HX | \Ar(\chi) \le D}. 
\] 
Here, the condition ``$\Ar(\chi) \le D$'' is defined by 
$\Ar(\phi^{\ast}\chi) \le \phibarastD$ 
for every $\phi:C\to X\in \CuX$ 
using the Artin conductor $\Ar(\phiast\chi)$ 
of the character 
$\phiast\chi:\pi_1(C)^{\ab} \to \QZ$ 
(which is corresponding to $\phiast\chi \in H^1(C,\QZ)$)  
(\cite{SerCL}, Chap.~IV, Sect.~3).  
We fix $\mathbb{C}^{\times} \simeq (\Qlbar)^{\times}$.  
One can consider $\chi \in (\piXDab)^{\vee}$ 
as a $\Qlbar$-representation (with finite image) 
\[
\xymatrix@C=5mm{
\piX \ar@{->>}[r]& \piXD \ar@{->>} [r]& \piXDab \ar[r]^-{\chi} & \QZ \ar@{^{(}->}[r] &\GL_1(\Qlbar)
}
\] 
of $\piX$, where $\piXD \surj \piXDab$ is the quotient map. 
We denote by $\F_{\chi}$ the corresponding $\Qlbar$-sheaf on $X$ (by \eqref{eq:l-adic}). 
It satisfies $\Sw(\F_{\chi}) \le D$  by Lem.~\ref{lem:bdd}.
For every $\phi:C\to X \in \CuX$,  
we have 
\begin{equation}
	\label{eq:Ar}
	\Ar(\phiast\chi) = \Ar(\phiast \F_{\chi})\ \mbox{defined in Rem.~\ref{rem:Ar}}.
\end{equation}	
Take $E \in \DivX$ such that $Z = |E|$ and we denote by   
$E_{\red}$ the reduced divisor associated to $E$. 
Putting $D' = E_{\red} + D$,  
we obtain 
\begin{align*}
		\Ar(\phi^{\ast}\chi) &=  \Ar(\phi^{\ast}\F_{\chi}) \quad (\mbox{from \eqref{eq:Ar}})\\
		&= \epsilon(\phi^{\ast}\F_{\chi}) + \Sw (\phi^{\ast}\F_{\chi}) \quad (\mbox{from \eqref{eq:epsilon}}) \\
		& \le \phibarast(E_{\red}) + \phibarast(D)\\
		&= \phibarast(D'). 
\end{align*}  
This implies that 
the image of $\psi$ defined in \eqref{eq:H1} 
is contained in $\filDpHX$. 
By taking the dual of $\psi: (\piXDab)^{\vee} \inj \filDpHX$, we have a surjective homomorphism 
$\psi^{\vee}:(\filDpHX)^{\vee} \surj  \piXDab$ using the Pontryagin duality theorem.  
(In fact, the group $(\filDpHX)^{\vee}$ is denoted by $\pi_1^{ab}(\Xbar,D')$ in \cite{KS}.) 
The structure map $X \to \Spec(k)$ 
induces a commutative diagram 
\[
  \xymatrix@C=15mm{
      (\filDpHX)^{\vee} \ar@{->>}[r]^-{\psi^{\vee}}\ar[d] & \piXDab\ar[d]^-{\varphi^{\ab}} \\
      H^1(\Spec(k),\Q/\Z)^{\vee}\ar[r]^-{\simeq} & \pi_1(\Spec(k))^{\ab}.  
  }
\]
Here,  the left vertical map is given by $H^1(\Spec(k),\QZ) \to \HX$.  
As $\Xbar$ is normal, 
the kernel 
of the left vertical map (which is $\pi_1^{ab}(\Xbar,D')^0$ in the sense of \cite{KS}) 
is finite by Cor.~1.2 in \cite{KS}.  
Our claim $\# \piXDabg < \infty$ follows from this and the commutative diagram above. 
\end{proof}

\begin{proof}[Proof of Cor.~\ref{cor:finite}]
As in the proof of Lem.~4.4 in \cite{HH}, 
by induction on $n$, 
we also obtain the finiteness of the quotient 
$\piXD^0/(\piXD^0)^{(n)}$  for any $n\in \Z_{\ge 1}$, 
where $(\piXD^0)^{(n)}$ 
is the $n$-th commutator subgroup 
(\cite{HH}, Def.~4.2).  
Starting from this, 
the same proofs of Thm.~4.5 (ii) and Thm.~4.6 (ii) in \cite{HH} 
work and the assertions follow.  
\end{proof}


\providecommand{\bysame}{\leavevmode\hbox to3em{\hrulefill}\thinspace}
\providecommand{\MR}{\relax\ifhmode\unskip\space\fi MR }
\providecommand{\MRhref}[2]{%
  \href{http://www.ams.org/mathscinet-getitem?mr=#1}{#2}
}
\providecommand{\href}[2]{#2}

\bigskip\noindent
Toshiro Hiranouchi \\
Department of Mathematics, Graduate School of Science, Hiroshima University
1-3-1 Kagamiyama, Higashi-Hiroshima, 739-8526 JAPAN\\
Email address: \texttt{hira@hiroshima-u.ac.jp}

\end{document}